\documentclass[11pt,reqno]{article}

\usepackage{caption,amsmath,amssymb,amsfonts,amsthm,latexsym,graphicx,multirow,color,hyperref,enumerate}
\usepackage{graphicx}
\usepackage{subcaption}
\usepackage{tikz}
\usetikzlibrary{arrows.meta, decorations.markings}
\usetikzlibrary{positioning}

\newtheorem{theorem}{Theorem}[section]

\newtheorem{lemma}[theorem]{Lemma}
\newtheorem{proposition}[theorem]{Proposition}

\theoremstyle{definition}
\newtheorem{definition}[theorem]{Definition}

\newtheorem*{notation}{Notation}

\numberwithin{equation}{section}
\numberwithin{table}{section}

\def\a{\alpha} \def\G{\Gamma}
 
\newcommand{\Aut}{\mathrm{Aut}}
\def\Cay{\hbox{\rm Cay}}

\def\f{\noindent}
\def\mz{{\mathbb Z}}

\begin{document}
	\title{On $m$-partite oriented semiregular representations of finite groups}
	
		\author{
		Jia-Li Du\thanks{School of Mathematical Sciences, Nanjing Normal University, Nanjing, 210023, P.R. China}
		\thanks{ Ministry of Education Key Laboratory of NSLSCS, Nanjing, 210023, P.R. China}}
	
	\date{}
	\maketitle

	\begin{abstract}
		The study of ORR was inspired by L\'{a}zsl\'{o} Babai in 1980 when he asked a question: Which [finite] groups admit an oriented graph as a DRR? And it has been solved by Joy Morris and Pablo Spiga through a series of papers in 2018. In this paper, we will extend the concept of ORR	 
		to $m$-partite oriented graphs for $m\geq 2$. We say that a finite group $G$ admits an \emph{$m$-partite oriented semiregular representation} ($m$-POSR) if there exists an $m$-partite oriented graph $\G$ such that its automorphism group is isomorphic to $G$ and acts semiregularly with the $m$ orbits giving the partition.
		 Moreover, if $\G$ is regular,
		that is, each vertex has
		the same in- and out-valency,
		it can be viewed as the oriented version of an $m$-Haar graph of $G$ and we call $\G$ is an \emph{$m$-Haar oriented representation} ($m$-HOR) of $G$. 	
		Our main result is a complete classification of finite groups $G$ without $m$-HORs or $m$-POSRs for $m\geq 2$.\medskip

		\f {\bf Keywords:} Semiregular group, regular representation, $m$-Cayley digraph, oriented graph, ORR.
		\medskip
		
		\f {\bf 2010 Mathematics Subject Classification:} 05C25, 20B25.
	\end{abstract}

\section{Introduction}

The study of graphical representation can be traced back to a question raised by K$\ddot{o}$nig~\cite{Konig} in 1936: whether a given group can be represented as the automorphism group of a graph. A group $G$ is said to admit a {\em{\rm(}di{\rm)}graphical
regular representation} (GRR or DRR) if there exists a (di)graph $\G$ such that its automorphism group is isomorphic to $G$ and acts regularly on the vertex set, and $\G$ is called a GRR or DRR of $G$ corresponding. Babai~\cite{Babai} proved that every finite group admits a DRR except for $Q_8$, $\mz_2^2$, $\mz_2^3$, $\mz_2^4$ and $\mz_3^2$. Compared with DRRs, GRRs proves to be significantly more challenging. For some of the most influential papers along the way, we refer to~\cite{Godsil,Hetzel,Imrich,ImrichWatkins2,NowitzWatkins1,NowitzWatkins2}.
	A digraph is said to be {\em oriented} if there is at most one arc between two vertices. We say that a group $G$ admits an {\em oriented regular representation} (ORR) if there exists an oriented graph such that its automorphism group is isomorphic to $G$ and acts regularly on the vertex set. The concept was first proposed in 1980, when Babai~\cite{Babai} raised the following question: Which [finite] groups admit an oriented graph as a DRR? In 2018, Morris and Spiga~\cite{MorrisSpiga1,MorrisSpiga2,Spiga} completed  the classification of the finite groups admitting ORRs.\smallskip
	
	A group $G$ is said to admit an {\em oriented $m$-semiregular representation} (O$m$SR)
	if there exists an oriented graph $\G$ such that its automorphism group is isomorphic to $G$ and acts semiregular on $V(\G)$ with $m$ orbits
	 for some positive integer $m$, and $\G$ is called an O$m$SR of $G$.
	If $\G$ is regular, we say $G$ admits a \emph{regular oriented $m$-semiregular representation} (regular O$m$SR), and $\G$ is a regular O$m$SR of $G$. Clearly, O$1$SRs are ORRs. In 2024, Du, Feng and Bang \cite{DFB} gave the classification of 
finite groups admitting O$m$SRs and regular O$m$SRs with $m\geq 2$. Now, we extend the definition to $m$-partite oriented graphs for $m\geq 2$. We say that a group $G$ admits an \emph{$m$‐partite oriented semiregular representation} ($m$-POSR) if there exists an $m$-partite oriented graph $\G$ such that  its automorphism group is isomorphic to $G$ and acts semiregularly with the $m$ orbits giving the partition, and $\G$ is called an $m$-POSR of $G$. Moreover, if $\G$ is regular,
then $\G$ is called a \emph{regular $m$-partite oriented graphical representation} (regular $m$-POSR) of $G$. \smallskip	
	
		On the other hand, a regular $m$-POSR of $G$ can be viewed as the oriented version of an $m$-Haar graphical representation of $G$ (which will be introduced later). Extending the well-studied concept of GRR or DRR to bipartite (di)graphs, a \emph{Haar (di)graphical representation} (HGR or HDR) of a group $G$ is a bipartite (di)graph whose automorphism group is isomorphic to $G$ and acts semiregularly with the orbits giving the bipartition. Notice that Haar graphs are necessarily regular. Recently, Morris and Spiga give the classification of HGRs of finite groups in \cite{MorrisSpiga3}. The $m$-Haar graphical
representation ($m$-HGR) as a natural generalization of
HGR to $m$-partite graphs for $m\geq 2$ was introduced in~\cite{DFXY}. Formally, a graph $\G$ is called an \emph{$m$-Haar graphical representation} ($m$-HGR) of a group $G$ if $\G$ is regular and $m$-partite such that its automorphism group is isomorphic
to $G$ and acts semiregularly on the vertex set with orbits giving the $m$-partition. In the same paper~\cite{DFXY}, the complete classification of finite groups without $m$-HGRs was given. For digraphs, the complete classification of 
$m$-HDRs was established by Du, Feng and Spiga~\cite{DFS2020-1,DFS2022}.\smallskip

 Now, we consider Haar graphs of oriented case and extend it to $m$-partite graphs. An oriented graph $\G$ is a {\em Haar oriented graph} of a group $G$ if $\G$ is regular, bipartite and its automorphism
 group has a subgroup
 isomorphic to $G$ and acts semiregularly with
 the orbits giving the bipartition.
 In the following, we give the definitions of the $m$-Haar oriented graph and the $m$-Haar oriented representation of finite groups.
	
		\begin{definition}\label{DefmHOR}
		For an integer $m\geq2$, an oriented graph $\Gamma$ is called an \emph{$m$-Haar oriented graph} of a group $G$ if $\Gamma$ is regular and $m$-partite such that the automorphism group $\Aut(\Gamma)$ has a subgroup isomorphic to $G$ and acts semiregularly on the vertex set with orbits giving the $m$-partition. Moreover, if $\Aut(\G)$ is isomorphic to $G$,
		then $\G$ is called an
		\emph{$m$-Haar oriented representation} ($m$-HOR) of $G$.
		\end{definition}
	
	By the definition, we know that an $m$-HOR is an $m$-POSR.
	But an $m$-POSR is not an $m$-HOR unless it is regular. In this paper, we consider $m$-HORs and $m$-POSRs for finite group with $m\geq 2$, and the main results of this paper as follows:
	
	\begin{theorem}\label{theo=main}
		Let $m\geq2$ be an integer and let $G$ be a finite group. Then $G$ admits an $m$-HOR, unless
		\begin{enumerate}
			\item[\rm(1)] $m=2$, and $G$ is isomorphic to either $Q_8$, $\mz_2^2$,
			$\mz_2^3$, $\mz_2^4$ or $\mz_n$ with $1\leq n\leq 5$;
			\item[\rm(2)]  $m=3$, and $G$ is isomorphic to $\mz_1$ or $\mz_2$;
			\item[\rm(3)]  $4\leq m\leq 6$, and $G$ is isomorphic to $\mz_1$.
		\end{enumerate}
	\end{theorem}

	\begin{theorem}\label{theo=main1}
	Let $m\geq2$ be an integer and let $G$ be a finite group. Then $G$ admits an $m$-POSR except $m=2$ and $G$ is isomorphic to either $\mz_3$, $\mz_2^2$,
	or $\mz_2^3$.
\end{theorem}

	This paper is organized as follows. In Section 2, we set up notations and review known results which will be used in this paper. In Section 3, we prove Theorem \ref{theo=main} when $G$ has at most three generators, and in Section 4, we prove Theorem~\ref{theo=main} when $G$ has at least four generators. The proof of Theorem~\ref{theo=main1} is given in Section 5.
	
	\section{ORRs, $m$-Cayley digraphs and notations}

	Let $G$ be a group. Denote by $d(G)$ the smallest size of a generating set of $G$. For $g\in G$, denote the order of $g$ by $|g|$. If $G$ acts on a set $\Omega$, then for $\omega \in \Omega$, the \emph{stabilizer of $\omega$} in $G$ is the subgroup of $G$ that consists of elements fixing $\omega$, denoted by $G_\omega$. We say that $G$ is \emph{semiregular} on $\Omega$ if $G_\omega=1$ for all $\omega \in \Omega$. Moreover, $G$ is \emph{regular} if it is semiregular and transitive.\smallskip
	
	For convenience in the subsequent discussion, we introduce some terminology and notations related to $m$-Cayley digraphs of the group $G$.
	For a subset $T$ of $G$, denote
	\[
	T^{-1}=\{t^{-1}:t\in T\}.
	\]
	Now let $m$ be a positive integer, and let $(T_{i,j})_{m\times m}$ be an $m\times m$ matrix with entries $T_{i,j}\subseteq G$ such that $1\notin T_{i,i}$ for all $i\in\{1,\dots,m\}$.
	Then the \emph{$m$-Cayley digraph} $\Cay(G,(T_{i,j})_{m\times m})$ of $G$ with \emph{connection matrix} $(T_{i,j})_{m\times m}$ is the digraph with vertex set
	\[
	G\times\{1,\dots,m\}=\bigcup_{i\in\{1,\dots,m\}}G_i,
	\]
	where $G_i=G\times\{i\}=\{g_i: g\in G\}$ and $g_i=(g,i)$, and arc set
	\begin{equation}\label{EqDef}
		\bigcup_{i,j\in\{1,\dots,m\}}\big\{\left( g_i, (tg)_j\right) :g\in G,\,t\in T_{i,j}\big\}.
	\end{equation}
	Note here that an $m$-Cayley digraph $\Cay(G,(T_{i,j})_{m\times m})$ is an oriented $m$-Cayley digraph if and only if $T_{i,j}\cap T_{j,i}^{-1}=\emptyset$ for all $i,j\in \{1,\cdots,m\}$. Without loss of generality, let $\Gamma=\Cay(G,(T_{i,j})_{m\times m})$ be an $m$-Cayley digraph of a group $G$ with respect to $T_{i,j} ~(i,j\in  \{1,\cdots,m\})$. For any element $g\in G$, the right multiplication $R(g)$ (i.e., mapping each vertex $x_i\in G_i$ to $(xg)_i\in G_i$ for all $i\in\{1,\cdots,m\}$) is an automorphism of $\G$, and $R(G)$ is a semiregular group of automorphisms of $\G$ with $G_i$ as orbits, where
	$$R(G):=\{ R(g)\ |\ g\in G\}.$$

	\begin{notation}
		Throughout the paper, $g_i$ is a concise notation for $(g,i)$ if $g$ is an element of a group and $i$ is an positive integer (we will still use $(g,i)$ sometimes if it makes the context clearer). Similarly, for a set $S$ of elements, $S\times\{i\}$ can be concisely denoted as $S_i$.
	\end{notation}
	
	The $1$-Cayley digraph of $G$ with connection matrix $(T)$ is exactly the Cayley digraph of $G$ with connection set $T$. A regular $m$-Cayley digraph $\Cay(G,(T_{i,j})_{m\times m})$ with $m\geq2$ and $T_{1,1}=\dots=T_{m,m}=\emptyset$ is called an \emph{$m$-Haar digraph}. Moreover, if $T_{i,j}=T_{j,i}^{-1}$ or $T_{i,j}\cap T_{j,i}^{-1}=\emptyset$ for each $i,j\in \{1,\cdots,m\}$, then $\Cay(G,(T_{i,j})_{m\times m})$ is called an \emph{$m$-Haar graph} or \emph{$m$-Haar oriented graph}, respectively.
	Hence Definition~\ref{DefmHOR} leads to the following fact.
	
	\begin{lemma}
		For a group $G$, the $m$-HORs of $G$ are the $m$-Haar oriented graphs of $G$ whose automorphism group is isomorphic to $G$.
	\end{lemma}
	
	In the following, we give the results of which will be used in the later.
	
	\begin{proposition}\label{prop=Haar}{\rm \cite[Lemma 3.2]{DFS2020-1}}
		Let $G$ be a finite group and let $\sigma $ be a permutation of $G$.
		Denote by $\sigma'$ the permutation of $G_1\cup G_2$ mapping
		$g_i$ to $g^{\sigma}_i$ for each $g\in G$ and $i=1,2$.
		Then the following are equivalent.
		\begin{enumerate}
			\item[\rm(1)] $\sigma\in \Aut(\Cay(G,S))\cap\Aut(\Cay(G,T))$,
			\item[\rm(2)]  $\sigma'\in\Aut(\Cay(G,(T_{i,j})_{2\times 2}))$ with $T_{1,2}=S$, $T_{2,1}=T$ and $T_{1,1}=T_{2,2}=\emptyset$.
		\end{enumerate}
	\end{proposition}

	\begin{proposition}\label{prop=OmSR2}{\rm \cite[Lemma 3.1]{DKY}}
		Let $m\geq 2$ be an integer and let $G$ be a finite group.
Let $\G=\Cay(G,(T_{i,j})_{m\times m})$ be a connected $m$-Cayley digraph of $G$ with $i,j \in \{1,\cdots,m\}$. For $A=\Aut(\G)$, if $A$ stabilizes $G_i$ for all $i \in \{1,\cdots,m\}$ and there exist $u_i \in G_{i}$ such that $A_{u_i}$ fixes $\G^+(u_i)$ pointwise for all  $i \in\{1,\cdots,m\} $, then $A=R(G)$.
	\end{proposition}

	The next proposition is the result about the classification of $m$-HOR for finite groups generated by at most two elements with $m\geq 2$.
	
	\begin{proposition}\label{prop=m-HOR}{\rm \cite[Theorem 1.1]{DKY}}
		Let $m\geq 2$ be an integer and let $G$ be a finite group with $d(G)\leq 2$.
		Then one of the following holds:
		\begin{enumerate}
			\item [\rm(1)] $G$ has an $m$-HOR of valency two;
			\item [\rm(2)] $m=2$ and $G\cong  \mz_2^2$, $Q_8$ or $\mz_n$ with $1\leq n\leq 5$;
			\item [\rm(3)] $m=3$ and $G\cong \mz_1$ or $\mz_2$;
			\item [\rm(4)] $4\leq m\leq 6$ and $G\cong \mz_1$.
		\end{enumerate}
	\end{proposition}
	
To end this section, we turn to the analysis of groups with the aim of simplifying the proof of Theorems~\ref{theo=main} and~\ref{theo=main1}.
Let $m\geq2$ be an integer, let $G$ be an elementary abelian 2-group and let $\Cay(G,(T_{i,j})_{m\times m})$ be a regular oriented $m$-Cayley graph over $G$. Then $T_{i,j}\cap T_{j,i}^{-1}=\emptyset$ for all $i,j\in \{1,\cdots,m\}$. In particular, $T_{i,i}\cap T_{i,i}^{-1}=\emptyset$. Since $G$ is an elementary abelian 2-group, we have $T_{i,i}=\emptyset$ and so $\Cay(G,(T_{i,j})_{m\times m})$ is an $m$-Haar oriented graph over $G$. Therefore, O$m$SRs are $m$-POSRs for elementary abelian 2-groups. 
By \cite[Theorem 1.2]{DFB}, an elementary abelian 2-group has no $m$-POSR if and only if $m=2$, and $G\cong\mz_2^2$ or $\mz_2^3$. By \cite[Theorem 1.1]{DFB}, an elementary abelian 2-group has no $m$-HOR if and only if $m=2$, and $G\cong\mz_2^2$, $\mz_2^3$ or $\mz_2^4$; $2\leq m\leq 3$, and $G\cong\mz_2$; or $2\leq m\leq 6$ and $G\cong\mz_1$. Therefore, in the rest of this paper, we just need to consider non-elementary abelian $2$-groups.\smallskip
	
Note that a group is an elementary abelian $2$-group if and only if each minimal generating set consists of involutions. Let $G$ be a non-elementary abelian $2$-group with $d(G) = t$ and let $\{h_1, h_2,\cdots,h_t\}$ be a generating set of $G$ such that $|h_i|\geq |h_{i+1}|$ for each $1\leq i\leq t-1$. Then $|h_1|\geq 3$ as $G$
	is a non-elementary abelian $2$-group.
	
	\section{Groups with small generating set}
 In this section, we address the case where the group 
$G$ satisfies $d(G)\leq 3$.
	
	\begin{lemma}\label{lem=m-HOR}
		Let $m\geq 2$ be an integer and let $G$ be a non-elementary abelian $2$-group with $d(G)\leq 3$. Then $G$ has an $m$-HOR
		except $m=2$ and $G\cong Q_8$ or $\mz_n$ with $3\leq n\leq 5$.
	\end{lemma}
	
	\begin{proof}
		Let $G$ be a finite group with $d(G)\leq 2$. Suppose $G$ has no $m$-HOR.
		Proposition~\ref{prop=m-HOR} implies that $m=2$ and $G$ is one of the following groups $Q_8$, $\mz_3$, $\mz_4$ or $\mz_5$.
	 By Magma~\cite{magma},
		$Q_8$, $\mz_3$, $\mz_4$ and $\mz_5$ has no $2$-HOR.\smallskip

Let $G$ be a finite group with $d(G)=3$. By~\cite[Lemma 2.6]{MorrisSpiga1}, $G$ is a generalized dihedral group or $G=\left\langle h_1,h_2,h_3\right\rangle$ with $|h_1|\geq |h_2|\geq |h_3|\geq 3$. If $G$ is a generalized dihedral group, then $G=\left\langle H,\tau\right\rangle$ with $H$ is abelian, $|\tau|=2$ and $h^{\tau}=h^{-1}$ for each $h\in H$. Since $G$ is a non-elementary abelian 2-group, we have that $H$ has a generating set without involution. Therefore, for both cases, $G$ has a generating set $\{a,b,c\}$ such that $|a|\geq |b|\geq |c|$ with $|a|\geq |b|\geq 3$. Let 
$$S=\{1,a,a^{-1},b\},~T=\{b,c,ba,cb\}{~~\rm and ~}~L=\{b,c,ab,a^{-1}b\}$$ 
be three subsets of $G$. It is easy to see that $S\cap T^{-1}=S\cap L^{-1}=\emptyset$. In the following, we divide the proof into four cases according to the value of $m$.
		
\medskip
\f {\bf  Case 1: $m=2$}
\medskip
		
Let $\G_2=\Cay(G,(T_{i,j})_{2\times 2})$ with $T_{1,2}=S$ and $T_{2,1}=T$. Then $\G_2$ is a connected $2$-Haar oriented graph of $G$ with valency $4$. Let $A=\Aut(\G_2)$. By the construction of $\G_2$, we have
		$$\G_2^+(1_1)=\{1_2,a_2,(a^{-1})_2,b_2\} {\rm ~and~} \G_2^+(1_2)=\{b_1,c_1,(ba)_1,(cb)_1\}.$$
Therefore,
		\begin{align*}
			\G_2^+(a_2)&=\{(ba)_1,(ca)_1,(ba^2)_1,(cba)_1\},\\
			 \G_2^+((a^{-1})_2)&=\{(ba^{-1})_1,(ca^{-1})_1,b_1,(cba^{-1})_1\},\\
			\G_2^+(b_2)&=\{(b^2)_1,(cb)_1,(bab)_1,(cb^2)_1\},\\
			\G_2^+(b_1)&=\{b_2,(ab)_2,(a^{-1}b)_2,(b^2)_2\},\\
			\G_2^+(c_1)&=\{c_2,(ac)_2,(a^{-1}c)_2,(bc)_2\},\\
			\G_2^+((ba)_1)&=\{(ba)_2,(aba)_2,(a^{-1}ba)_2,(b^2a)_2\}, {\rm and }\\
			\G_2^+((cb)_1)&=\{(cb)_2,(acb)_2,(a^{-1}cb)_2,(bcb)_2\}.
		\end{align*}
		It is easy to see that $\G_2^+(1_2)\cap \G_2^+(a_2)=\{(ba)_1\}$,
		$\G_2^+(1_2)\cap \G_2^+((a^{-1})_2)=\{b_1\}$, $\G_2^+(1_2)\cap \G_2^+(b_2)=\{(cb)_1\}$ and $\G_2^+(b_2)\cap \G_2^+((a^{-1})_2)=\G_2^+(b_2)\cap \G_2^+(a_2)=\emptyset$.
		It implies that $1_2$ is the unique out-neighbor of $1_1$ that has a common out-neighbor with each vertex $u\in \G_2^+(1_1)$.
		Therefore, $A_{1_1}$ fixes $1_2$. Moreover, $A_{1_1}$ fixes $c_1$ as it is the unique out-neighbor of $1_2$ which is not an out-neighbor of other vertex
		in $\G_2^+(1_1)$.
		For $\G_2^+(1_2)$, we have $\G_2^+(b_1)\cap \G_2^+(c_1)=\G_2^+((ba)_1)\cap \G_2^+((cb)_1)=\emptyset$, that is, there is no vertex $u$
		in $\G_2^+(1_2)$ such that $u$ has a common out-neighbor with each vertex $v\in \G_2^+(1_2)$.
		 Therefore, $A$ stabilizes $G_1$ and $G_2$ respectively.\smallskip

		Recall taht $R(G)\leq A$. Then $A$ has exactly two orbits $G_1$ and $G_2$ and hence $A=R(G)A_{1_1}=R(G)A_{1_2}$. Since $G$ is semiregular and $A_{1_1}\leq  A_{1_2}$, $A_{1_2}=A_{1_2}\cap A=A_{1_2}\cap R(G)A_{1_1}=A_{1_1}$ holds and so $A_{g_1}=A_{g_2}$ for each $g\in G$. We will prove $\{g_1,g_2\}$ is a block of $A$ acting on $V(\G_2)$ for each $g\in G$. Let $\sigma\in A$. Suppose $\{g_1,g_2\}^{\sigma}\cap \{g_1,g_2\}\not=\emptyset$. Since $A_{g_1}=A_{g_2}$ and $A$ stabilizes $G_i~(i=1,2)$,
		we have
		$$g_1^{\sigma}=g_1 \mbox{~and ~}g_2^{\sigma}=g_2.$$
		This means that $\{\{g_1,g_2\}\ |\ g\in G\}$ is a block system of $A$ on $V(\G_2)$, and hence there is a permutation $\overline{\sigma}$ of $G$ satisfying $(g^{\overline{\sigma}})_1=(g_1)^{\sigma}$ and $(g^{\overline{\sigma}})_2=(g_2)^{\sigma}$ for each $g\in G$. By Proposition~\ref{prop=Haar}, $\overline{\sigma}\in \Aut(\Cay(G,S))\cap \Aut(\Cay(G,T))$. \smallskip
		
		Now, we are in the position to prove $A_{1_1}=1$. Let $\a\in A_{1_1}$. Then there exists $\overline{\a}\in $ $\Aut(\Cay(G,S))_1$ $\cap \Aut(\Cay(G,T))_1$ by the paragraph above. It implies that $S^{\overline{\a}}=S$ and $T^{\overline{\a}}=T$.
In particular, $\{b\}^{\overline{\a}}=(S\cap T)^{\overline{\a}}=S\cap T=\{b\}$, that is, $\overline{\a}$ fixes $b$.
		 Then  $\{b,ba\}^{\overline{\a}}=(S^2\cap T)^{\overline{\a}}=S^2\cap T=\{b,ba\}$, which implies $\overline{\a}$ fixes $ba$.
		Recall that $\a \in A_{1_1}$ fixes $c_1$. It means that $\overline{\a}$ fixes $c$.
		Thus, ${\overline{\a}}$ fixes $b,c,ba$ pointwise, and so ${\overline{\a}}$ fixes $\left\langle b,c,ba\right\rangle =G$ pointwise. Thus, $\overline{\a}=\a=1$.
		It follows that $A_{1_1}=1$ and so $A=R(G)$. Therefore, $\G_2$ is a $2$-HOR of $G$.		
	\medskip

		\f {\bf Case 2: $m=3$}
		\medskip
		
		Let $\G_3=\Cay(G,(T_{i,j})_{3\times 3})$ with $T_{1,2}=T_{3,1}=T_{3,2}=S$, $T_{2,1}=T_{2,3}=T$ and $T_{1,3}=L$. Since $S\cap T^{-1}=S\cap L^{-1}=\emptyset$, we have that $\G_3$ is a connected $3$-Haar oriented graph of $G$ with valency $8$. Let $A=\Aut(\G_3)$.
		By the construction of $\G_3$, we have
		\begin{align*}
			\G_3^+(1_1)&=S\times \{2\}\cup L\times \{3\}=\{1_2,a_2,(a^{-1})_2,b_2,b_3,c_3,(ab)_3,(a^{-1}b)_3\},\\
			\G_3^+(1_2)&=T\times \{1\}\cup T\times \{3\}=\{b_1,c_1,(ba)_1,(cb)_1,b_3,c_3,(ba)_3,(cb)_3\}, {\rm and }\\
			\G_3^+(1_3)&=S\times \{1\}\cup S\times \{2\}=\{1_1,a_1,a^{-1}_1,b_1,1_2,a_2,a^{-1}_2,b_2\}.
		\end{align*}
	Now, we consider the sub-digraphs induced by $\G_3(1_i)$ with $i\in \{1,2,3\}$.	With $[\G_3^{+}(1_1)]$, $[\G_3^{+}(1_2)]$ and $[\G_3^{+}(1_3)]$ depicted in Figure~\ref{Fig1} (where the dished arcs means possible arcs), it is straightforward to observe that $[\G_3^{+}(1_3)]\ncong [\G_3^{+}(1_i)]$ for $i\in\{1,2\}$ as $1_1$ has out-valency $4$ in $[\G_3^{+}(1_3)]$ and every vertex has out-valency at most 3 in $[\G_3^{+}(1_1)]$ and $[\G_3^{+}(1_2)]$. Consequently, $A$ stabilizes $G_3$, and so stabilizes $G_1\cup G_2$.
		Since $[G_1\cup G_2]\cong \G_2$
		is a $2$-HOR of $G$, $A_{1_1}$ fixes $G_1\cup G_2$ pointwise. In particular, $A$ stabilizes $G_i$ for each $i\in\{1,2,3\}$.
		Now, to prove $\G_3$ is a $3$-HOR of $G$, we just need to prove $A_{1_1}=1$, that is, $A_{1_1}$ fixes $g_i$ pointwise
		for each $g\in G$ and $i\in \{1,2,3\}$. If there exists $\alpha\in A_{1_1}$ such that $(g_3)^{\a}=h_3$ for some $g,h\in G$, then $\a$ maps $\G_3^+(g_3)=(Sg)_1\cup (Sg)_2$ to $\G_3^+(h_3)=(Sh)_1\cup (Sh)_2$. Since $A_{1_1}$
		fixes $G_1\cup G_2$ pointwise, we have that $g_3=h_3$, and so $\alpha$ fixes $G_3$ pointwise.
		Thus, $A_{1_1}$ fixes $G_3$ pointwise. Therefore, $A_{1_1}$
		fixes $g_i$ for each $g\in G$ and $i\in \{1,2,3\}$, and hence $\G_3$ is a $3$-HOR of $G$.

		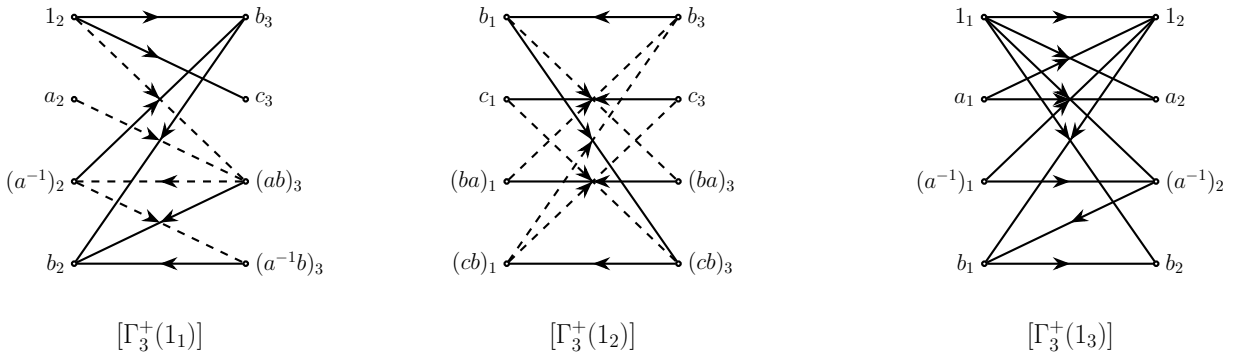
\begin{figure}[!ht]
			\vspace{1mm}
		
							\vspace{-2mm}
			\begin{tikzpicture}[node distance=0.7cm,thick,scale=0.6,every node/.style={transform shape},scale=1,
				midarrow/.style={decoration={markings, mark=at position 0.5 with {\arrow{Stealth}}}, postaction={decorate}}
				]			
				
				\node[circle](A0){};
				\node[below=of A0, circle,draw, inner sep=1pt, label=left:{\Large$1_2$}](A1){};
				\node[below=of A1](A11){};
				\node[below=of A11, circle,draw, inner sep=1pt, label=left:{\Large$a_2$}](A2){};
				\node[below=of A2](A21){};
				\node[below=of A21, circle,draw, inner sep=1pt, label=left:{\Large$(a^{-1})_2$}](A3){};
				\node[below=of A3](A31){};
				\node[below=of A31, circle,draw, inner sep=1pt, label=left:{\Large$b_2$}](A4){};
				\node[right=of A1](A00){};
				\node[right=of A00](A000){};
				\node[right=of A000](A0000){};
				\node[right=of A0000, circle,draw, inner sep=1pt, label=right:{\Large$b_3$}](B1){};
				\node[below=of B1](B10){};
				\node[below=of B10, circle,draw, inner sep=1pt, label=right:{\Large$c_3$}](B2){};
				\node[below=of B2](B21){};
				\node[below=of B21, circle,draw, inner sep=1pt, label=right:{\Large$(ab)_3$}](B3){};
				\node[below=of B3](B31){};
				\node[below=of B31, circle,draw, inner sep=1pt, label=right:{\Large$(a^{-1}b)_3$}](B4){};
				\node[right=of A4](A40){};
				\node[below=of A40](A41){};
				\node[right=of A41,label=below:{\LARGE$[\G_3^{+}(1_1)]$}]{};
				\draw[midarrow] (A1) -- (B1);
				\draw[midarrow] (A1) -- (B2);
				\draw[dashed,midarrow] (A1) -- (B3);
				\draw[dashed,midarrow](A2) -- (B3);
				\draw[midarrow] (A3) -- (B1);
				\draw[dashed,midarrow] (A3) -- (B4);
						
				\draw[midarrow] (B1) -- (A4);
					\draw[dashed,midarrow] (B3) -- (A3);
				\draw[midarrow] (B3) -- (A4);
				\draw[midarrow] (B4) -- (A4);
				
				\node[right=of B1](B11){};
				\node[right=of B11](B111){};
				\node[right=of B111](B1111){};
				\node[right=of B1111](B11111){};
				\node[right=of B11111](B111111){};
				\node[right=of B111111, circle,draw, inner sep=1pt, label=left:{\Large$b_1$}](C1){};
				\node[below=of C1](C10){};
				\node[below=of C10, circle,draw, inner sep=1pt, label=left:{\Large$c_1$}](C2){};
				\node[below=of C2](C21){};
				\node[below=of C21, circle,draw, inner sep=1pt, label=left:{\Large$(ba)_1$}](C3){};
				\node[below=of C3](C31){};
				\node[below=of C31, circle,draw, inner sep=1pt, label=left:{\Large$(cb)_1$}](C4){};
				\node[right=of C1](C00){};
				\node[right=of C00](C000){};
				\node[right=of C000](C0000){};
				\node[right=of C0000, circle,draw, inner sep=1pt, label=right:{\Large$b_3$}](D1){};
				\node[below=of D1](D11){};
				\node[below=of D11, circle,draw, inner sep=1pt, label=right:{\Large$c_3$}](D2){};
				\node[below=of D2](D21){};
				\node[below=of D21, circle,draw, inner sep=1pt, label=right:{\Large$(ba)_3$}](D3){};
				\node[below=of D3](D31){};
				\node[below=of D31, circle,draw, inner sep=1pt, label=right:{\Large$(cb)_3$}](D4){};
				\node[right=of C4](C40){};
				\node[below=of C40](C41){};
				\node[right=of C41,label=below:{\LARGE$[\G_3^{+}(1_2)]$}]{};
				\draw[dashed,midarrow] (C1) -- (D3);
				\draw[midarrow] (C1) -- (D4);
				\draw[dashed,midarrow] (C2) -- (D4);
				\draw[dashed,midarrow] (C3) -- (D1);
						\draw[dashed,midarrow] (C4) -- (D2);
				\draw[dashed,midarrow](C4) -- (D1);
				\draw [midarrow](D1) -- (C1);
				\draw [midarrow](D2) -- (C2);
				\draw [midarrow](D3) -- (C3);
				\draw[midarrow] (D4) -- (C4);
				\node[right=of D1](D11){};
				\node[right=of D11](D111){};
				\node[right=of D111](D1111){};
				\node[right=of D111](D1111){};
				\node[right=of D1111](D11111){};
				\node[right=of D11111](D111111){};
				\node[right=of D111111](D1111111){};
				\node[right=of D1111111, circle,draw, inner sep=1pt, label=left:{\Large$1_1$}](E1){};
				\node[below=of E1](E11){};
				\node[below=of E11, circle,draw, inner sep=1pt, label=left:{\Large$a_1$}](E2){};
				\node[below=of E2](E21){};
				\node[below=of E21, circle,draw, inner sep=1pt, label=left:{\Large$(a^{-1})_1$}](E3){};
				\node[below=of E3](E31){};
				\node[below=of E31, circle,draw, inner sep=1pt, label=left:{\Large$b_1$}](E4){};
				\node[right=of E1](E0){};
				\node[right=of E0](E00){};
				\node[right=of E00](E000){};
				\node[right=of E000](E0000){};
				\node[right=of E0000](E00000){};
				\node[right=of E000, circle,draw, inner sep=1pt, label=right:{\Large$1_2$}](F1){};
				\node[below=of F1](F11){};
				\node[below=of F11, circle,draw, inner sep=1pt, label=right:{\Large$a_2$}](F2){};
				\node[below=of F2](F21){};
				\node[below=of F21, circle,draw, inner sep=1pt, label=right:{\Large$(a^{-1})_2$}](F3){};
				\node[below=of F3](F31){};
				\node[below=of F31, circle,draw, inner sep=1pt, label=right:{\Large$b_2$}](F4){};
				\node[right=of E4](E40){};
				\node[below=of E40](E41){};
				\node[right=of E41,label=below:{\LARGE$[\G_3^{+}(1_3)]$}]{};
				\draw[midarrow] (E1) -- (F1);
				\draw[midarrow] (E1) -- (F2);
				\draw[midarrow] (E1) -- (F3);
				\draw[midarrow] (E1) -- (F4);
				\draw[midarrow] (E2) -- (F1);
				\draw[midarrow] (E2) -- (F2);
				\draw[midarrow](E3) -- (F1);
				\draw[midarrow] (E3) -- (F3);
				\draw[midarrow] (E4) -- (F4);
				\draw[midarrow] (F1) -- (E4);
				\draw [midarrow](F3) -- (E4);
			\end{tikzpicture}
	\caption{\small{The induced sub-digraphs  $[\G_3^{+}(1_i)]$ with $i\in\{1,2,3\}$}}\label{Fig1}
			\vspace{2mm}
		\end{figure}
		
\begin{figure}

					\vspace{-2mm}
	\centering
	\includegraphics[width=5cm]{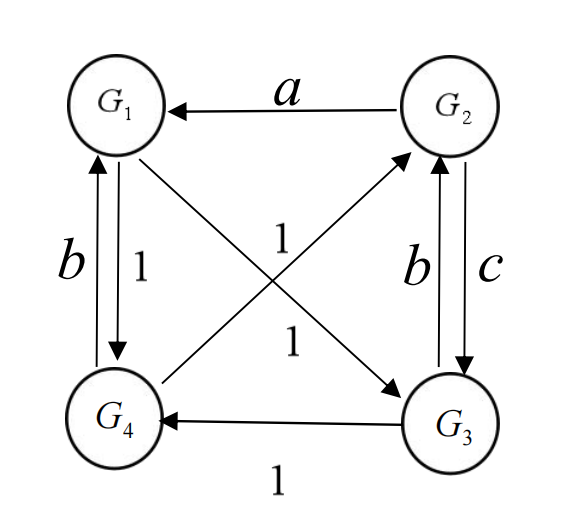}	 \includegraphics[width=8cm]{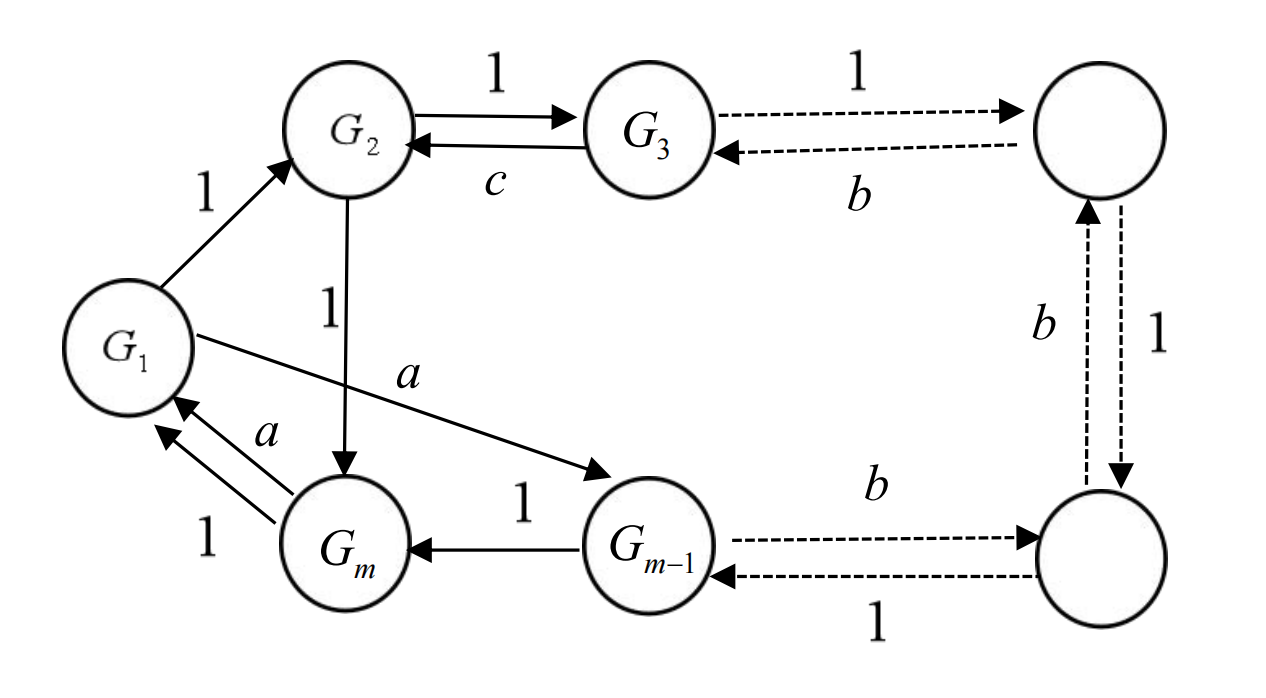}
	\caption{\small{ The $4$-Haar oriented graph and the $m$-Haar oriented graph $\G_m$~with $m\geq 5$}}\label{Fig2}
	\vspace{-0mm}
\end{figure}
 		
		\medskip
		\f {\bf Case 3: $m=4$}
		\medskip	
		
		Let $\G_4=\Cay(G,(T_{i,j})_{4\times 4})$ with $T_{1,4}=T_{1,3}=T_{3,4}=T_{4,2}=\{1\}$, $T_{2,1}=\{a\}$, $T_{3,2}=T_{4,1}=\{b\}$,
		$T_{2,3}=\{c\}$ and $T_{i,j}=\emptyset$ otherwise. Then $\G_4$ is a connected $4$-Haar oriented  graph of $G$ with valency $2$. The left graph in Figure~\ref{Fig2} might be of some help for understanding the structure of $\G_4$. Let $A=\Aut(\G_4)$.
		By the construction of $\G_4$, we have
		$$\G_4^{+}(1_1)=\{1_3,1_4\},~\G_4^{+}(1_2)=\{a_1,c_3\},~ \G_4^{+}(1_3)=\{1_4,b_2\}~{\rm and}~\G_4^{+}(1_4)=\{b_1,1_2\}.$$
		It is easy to see that the induced sub-digraph $[\G_4^{+}(1_1)]$ is the arc $(1_3,1_4)$, and $[\G_4^{+}(1_i)]$ is an union of two isolated vertices for each $i\in \{2,3,4\}$. It follows that $A$ stabilizes $G_1$, and so $A$ stabilizes $G_2\cup G_3\cup G_4$. Note that, for each $g\in G$,  $g_1$ has only in-neighbor in $G_2$, has only out-neighbor in $G_3$, and has both in- and out-neighbor in $G_4$. Then $A$ stabilizes $G_i$ for each $i\in\{1,\cdots,4\}$, that is, $A$ has $4$ orbits on $V(\G_4)$. Since the two out-neighbors of $1_i$ are in different orbits of $A$, we have that $A_{1_i}$ fixes $\G_4^{+}(1_i)$ pointwise for each $i\in \{1,\cdots,4\}$. By Proposition~\ref{prop=OmSR2}, $A=R(G)$ and so $\G_4$ is a $4$-HOR of $G$.

\medskip
\f {\bf Case 4: $m\geq5$}
\medskip	
		
		Let $\G_m=\Cay(G,(T_{i,j})_{m\times m})$ with
		\begin{eqnarray*}
			T_{i,i+1}&=&\{1\}\,\,\,\,\,\,\,\,\,\,\,\,\,\,\,\,\,\,\mbox{ for } i\in\mz_m;\\
			 T_{m,1}=T_{1,m-1}=\{a\},\,\,\,\,\,T_{2,m}=\{1\},\,\,T_{3,2}=\{c\},\,\,\,\, T_{i,i-1}&=&\{b\}\,\,\,\,\,\,\,\,\,\,\,\,\,\,\,\,\,\,\mbox{ for } i\not=1,2,3,m; \\
			 T_{i,j}&=&\emptyset\,\,\,\,\,\,\,\,\,\,\,\,\,\,\,\,\,\,\,\,\,\,\,\,\mbox{ otherwise.}
		\end{eqnarray*}
Then $\G_m$ is a connected $m$-Haar oriented graph of $G$ with valency $2$. The right graph in Figure~\ref{Fig2} might be of some help for understanding the structure of $\G_m$. Let $A=\Aut(\G_m)$. By the definition of $\G_m$, we have that $(g_1,g_2,g_{m})$ is a directed $3$-cycle for each $g\in G$. Again, by the definition of $\G_m$, there is no directed $3$-cycle through $g_i$ for each $i\in \{3,\cdots, m-1\}$ and $g\in G$.
So $A$ stabilizes $G_1\cup G_2\cup G_{m}$. Since both $g_1$ and $g_2$  have out-valency $1$, and $g_{m}$ has out-valency $2$ in  the induced sub-digraph $[G_1\cup G_2\cup G_{m}]$ for each $g\in G$, we have that $A$ stabilizes $G_{m}$.
Since all out-neighbors of $g_{m}$ are in $G_1$, $A$ stabilizes $G_1$ and so $A$ stabilizes $G_2$.
Recall that $T_{2,3}=\{1\}$ and $T_{2,j}=\emptyset$ for each $j\neq 3,m$. So $A$ stabilizes $G_3$. Repeating this step, $A$ stabilizes $G_{j}$ setwise for each $j\in \{4,\cdots,m-1\}$. It follows that $A$ stabilizes $G_i$  for each $i\in \{1,\cdots,m\}$.
Since $R(G)\leq A$ is transitive on $G_i$, $A$ exactly has $m$ orbits on $V(\G_m)$, that is, $G_i$ with $i\in\{1,\cdots,m\}$. \vskip0.1cm
	
For any $i \in \{1,\cdots,m-1\}$,  the two out-neighbors of $1_i$ belong to different orbits. So  $A_{1_i}$ fixes $\G_m^+(1_i)$ pointwise.
Note that $\G_m^+(1_{m})=\{1_1,a_{1}\}$ and $(1_{m},1_1,1_2)$ is a directed $3$-cycle. Since $|a|\geq 3$, there is no directed 3-cycle containing both $1_{m}$ and $a_1$. Thus, $A_{1_{m}}$ fixes $\G_m^+(1_{m})$ pointwise. By Proposition~\ref{prop=OmSR2}, $A=R(G)$ and hence $\G_m$ is an $m$-HOR of $G$. This completes the proof.
	\end{proof}
		
	\section{Groups with large generating set}

Throughout this section, let $G$ be a group with $d(G) = t \geq  4$. Let $\{h_1, h_2,\cdots,h_t\}$ be a generating set of $G$ such that $|h_1|\geq 3$. Let $\overline{h_1}=h_1$ and $\overline{h_i}=h_1\cdots h_i$  with $i\in \{2,\cdots,t\}$. Denote	$$S=\{1,h_i\mid 1\leq i\leq t\} {\rm ~and~} T=\{h_1h_t, \overline{h_i}\mid 1\leq i\leq t\}.$$
It is easy to see that $|S|=|T|=t+1$ and $S\cap T^{-1}=\emptyset$.

\begin{lemma}\label{lem=m-HOR1}
	Let $m\geq 2$ be an integer and let $G$ be a non-elementary abelian $2$-group with $d(G)\geq4$. Then
	$G$ has an $m$-HOR.
\end{lemma}

\begin{proof}
	Let $m\geq 2$ be an integer. In the following, we divide the proof into four cases according to the value of $m$.
	
	\medskip
\f	{\bf Case 1:} $m=2$
		\medskip
		
	Let  $\G_2=\Cay(G,(T_{i,j})_{2\times 2})$ with $T_{1,2}=S$ and $T_{2,1}=T$. Then $\G_2$ is a connected $2$-Haar oriented graph of $G$ with valency $t+1$. Let $A=\Aut(\G_2)$.
By the construction of $\G_2$, we have
$$\G_2^+(1_1)=S\times \{2\}=\{1,h_i\mid 1\leq i\leq t\}\times \{2\} {\rm ~and~} \G_2^+(1_2)=T\times \{1\}=
\{h_1h_t,\overline{h_i}\mid 1\leq i\leq t \}\times \{1\}.$$

\medskip
\f{\bf Claim:} $A$ stabilizes $G_1$ and $G_2$ respectively.
\medskip

To prove this claim, we will consider the induced sub-digraphs $[\G_2^+(1_i)]$ with $i\in \{1,2\}$. At first, we consider $[\G_2^+(1_1)]$. For each $i\geq 2$,  $\G_2^{+}(h_i,2)=(Th_i)\times\{1\}=\{h_1h_th_i,\overline{h_1}h_i,\cdots, \overline{h_t}h_i\}\times \{1\}$. In particular,
 $(\overline{h_i},1)=(\overline{h_{i-1}}h_i,1) \in  \G_2^{+}(h_i,2)$ for each $i\in \{2,\cdots,t\}$. Thus, $(\overline{h_i},1) $ is the common out-neighbor of $1_2$ and $(h_i,2)$ for each $i\in \{2,\cdots,t\}$.
For $i=1$,  $\G_2^{+}(h_1,2)=(Th_1)\times\{1\}=\{h_1h_th_1,h_1^2,\overline{h_2}h_1,\cdots,\overline{h_t}h_1\}$. It follows that $\G_2^{+}(1_2)\cap\G_2^{+}(h_1,2)=\emptyset$. Similarly, we have $\G_2^{+}(h_i,2)\cap \G_2^{+}(h_j,2) =\emptyset$ for each $i\neq j\in\{1,\cdots,t\}$.
 Hence, $1_2$ is the unique vertex in $\G_2^{+}(1_1)$ which has a common out-neighbor with exactly $t-1$ vertices in $\G_2^{+}(1_1)$. In particular, $A_{1_1}$ fixes $1_2$.
 \smallskip

Secondly, we consider $\G_2^{+}(1_2)=T\times \{1\}=\{h_1h_t,
 \overline{h_1},\cdots,\overline{h_t}\}\times \{1\}$. We will prove that for any vertex $u$ in $\G_2^+(1_2)$, there is at most two vertices $v$ in $\G_2^+(1_2)\setminus\{u\}$ such that $u$ and $v$ have a common out-neighbor.
Now, we will deal with $(h_1h_t,1)$ and $(\overline{h_1},1)$. Since $\G_2^{+}(h_1h_t,1)=(Sh_1h_t)\times \{2\}=\{h_1h_t,h_1^2h_t,h_2h_1h_t,\cdots,h_th_1h_t\}\times \{2\}$ and  $\G_2^{+}(\overline{h_i},1)=(S\overline{h_i})\times \{2\}=\{\overline{h_i},h_1\overline{h_i},\cdots,h_t\overline{h_i}\}\times \{2\}$ for each $i\in\{3,\cdots, t\}$,
  we have that $\G_2^{+}(h_1h_t,1)\cap \G_2^{+}(\overline{h_i},1)=\emptyset$ for each $i\in\{3,\cdots,t\}$.
    Therefore, there is at most two vertices in $\G_2^{+}(1_2)$ has a common out-neighbor with $(h_1h_t,1)$.
   Similarly, since $\G_2^{+}(\overline{h_1},1)=(Sh_1)\times \{2\}=\{h_1,h_1^2,h_2h_1,\cdots,h_th_1\}\times \{2\}$, we have $\G_2^{+}(h_1,1)\cap \G_2^{+}(\overline{h_i},1)=\emptyset$ for each $i\in \{3,\cdots,t\}$.
  Therefore, there is at most two vertices in $\G_2^{+}(1_2)$ has a common out-neighbor with $(h_1,1)$. 
  Now, we consider $(\overline{h_i},1)$ with $i\in\{2,\cdots,t\}$.
  Recall that 
  $\G_2^{+}(\overline{h_i},1) =(S\overline{h_i})\times \{2\}=\{\overline{h_i},h_1\overline{h_i},\cdots,h_t\overline{h_i}\}\times \{2\}$ for each $i\in\{2,\cdots,t\}$.
For $i\neq j\in\{2,\cdots,t\}$, we have
 	\begin{align*}
 	\G_2^{+}(\overline{h_i},1)\cap \G_2^{+}(\overline{h_j},1) &=
 	\begin{cases}
 		\{\overline{h_{j}}\}&\text{if }j=i+1 \text{ and } h_j\overline{h_{j-1}}=\overline{h_{j}},\\ 			 \{\overline{h_{i}}\}&\text{if }i=j+1 \text{ and } h_i\overline{h_{i-1}}=\overline{h_{i}},\\
 			\emptyset&\text{ otherwise}.
 	\end{cases}
 \end{align*}
 Therefore, there is at most two vertices in $\G_2^{+}(1_2)$ has a common out-neighbor with $(\overline{h_{i}},1)$ for each $i\in\{2,\cdots,t\}$. Since $t\geq 4$, we have $t-1\geq 3$. Recall that there is $t-1$ vertices has a common out-neighbor with $1_2$ in $\G_2^{+}(1_1)$. Thus, $[\G_2^+(1_1)]\cong [\G_2^+(1_2)]$ and so $A$ stabilizes $G_1$ and $G_2$, as claimed.\smallskip
  
  Now, we are ready to prove $\G_2$ is a $2$-HOR.
  As $R(G)\leq A$, $A$ has exactly two orbits $G_1$ and $G_2$ and hence $A=R(G)A_{1_1}=R(G)A_{1_2}$. Since $G$ is semiregular and $A_{1_1}\leq  A_{1_2}$, $A_{1_2}=A_{1_1}$ holds and so $A_{g_1}=A_{g_2}$. Similar as the arguments in Lemma~\ref{lem=m-HOR}, we will prove $\{g_1,g_2\}$ is a block for each $g\in G$, and then prove $A_{1_1}=1$. Let $\sigma\in A$. Suppose $\{g_1,g_2\}^{\sigma}\cap \{g_1,g_2\}\not=\emptyset$. Since $A_{g_1}=A_{g_2}$ and $A$ stabilizes $G_i~(i=1,2)$,
 we have
 $$g_1^{\sigma}=g_1 \mbox{~and ~}g_2^{\sigma}=g_2.$$
 	This means that $\{\{g_1,g_2\}\ |\ g\in G\}$ is a block system of $A$ on $V(\G_2)$, and hence there is a permutation $\overline{\sigma}$ of $G$ satisfying $(g^{\overline{\sigma}})_1=(g_1)^{\sigma}$ and $(g^{\overline{\sigma}})_2=(g_2)^{\sigma}$ for each $g\in G$. By Proposition~\ref{prop=Haar}, $\overline{\sigma}\in \Aut(\Cay(G,S))\cap \Aut(\Cay(G,T))$.
 	Let $\a\in A_{1_1}$. Then $\overline{\a}\in \Aut(\Cay(G,S))_1 \cap \Aut(\Cay(G,T))_1$. It follows that $S^{\overline{\a}}=S$
 and $T^{\overline{\a}}=T$.
 Moreover, $(S^i)^{\overline{\a}}=S^i$ for each $i\in\{1,\cdots,t\}$. In particular, $(S^i\cap T)^{\overline{\a}}=S^i\cap T$ for each $i\in\{1,\cdots,t\}$.
 Since $S\cap T=\{h_1\}$, $S^2\cap T=\{h_1h_t,h_1h_2\}$,
 and $S^i\cap T=\{\overline{h_i}\}$ for each $i\in\{3,\cdots,t\}$, we have  $\{h_1\}^{\overline{\a}}=\{h_1\}$, $\{h_1h_t,h_1h_2\}^{\overline{\a}}=\{h_1h_t,h_1h_2\}$, and  $\{\overline{h_i}\}^{\overline{\a}}=\{\overline{h_i}\}$  for each $i\in\{3,\cdots,t\}$. Therefore, $\overline{\a}$ fixes $h_i$ as $h_i=(\overline{h_{i-1}})^{-1}\overline{h_{i}}$ for each $i\in\{4,\cdots,t\}$. In particular, $\overline{\a}$ fixes $h_t$, and so  $\overline{\a}$ fixes $h_1h_t$.
 It follows that  $\overline{\a}$ fixes $\{h_1h_2\}$ and so  $\overline{\a}$ fixes $h_2$ and $h_3=(\overline{h_2})^{-1}\overline{h_3}$. Thus, $\overline{\a}$ fixes $S$ pointwise, it implies that $\overline{\a}$ fixes $\langle S\rangle=G$ pointwise. Thus $\overline{\a}=1$, and so $\a=1$.
 Therefore, $A_{1_1}=1$, and so $A=R(G)$. Hence, $\G_2$ is a $2$-HOR of $G$.

 	\medskip
 \f	{\bf Case 2:} $m=3$
 \medskip

  Let $L=\{h_1,h_2h_3,h_1h_i\mid 2\leq i\leq t\}$. Then $|L|=t+1$ and
 $S\cap L^{-1}=\emptyset$.
 Let  $\G_3=\Cay(G,(T_{i,j})_{3\times 3})$ with $T_{1,2}=T_{2,3}=T_{3,1}=S$, $T_{1,3}=T_{3,2}=L$ and $T_{2,1}=T$. Then $\G_3$ is a connected $3$-Haar oriented graph of $G$ with valency $2t+2$. Let $A=\Aut(\G_3)$.
 By the construction of $\G_3$, we have
\begin{align*}
\G_3^+(1_1)&=S\times \{2\}\cup L\times \{3\}=\{1,h_i\mid 1\leq i\leq t\}\times \{2\}\cup \{h_1,h_2h_3,h_1h_i\mid 2\leq i\leq t\}\times \{3\}, \\
\G_3^+(1_2)&=T\times \{1\}\cup S\times \{3\}=\{h_1h_t,\overline{h_i}\mid 1\leq i\leq t\}\times \{1\}\cup \{1,h_i\mid 1\leq i\leq t\}\times \{3\}, {\rm and }\\
\G_3^+(1_3)&=S\times \{1\}\cup L\times \{2\}=\{1,h_i\mid 1\leq i\leq t\}\times \{1\}\cup \{h_1,h_2h_3,h_1h_i\mid 2\leq i\leq t\}\times \{2\}.
\end{align*}
 In the following, we consider the  sub-digraphs $[\G_3^+(1_i)]$ induced by $\G_3^+(1_i)$ to prove $A$ stabilizes $G_i$ for each $i\in\{1,2,3\}$. First, we consider $[\G_3^+(1_2)]$. Since $T_{1,3}=L=\{h_1,h_2h_3,h_1h_i\mid 2\leq i\leq t\}$ and $t\geq 4$, we have that there is no arc from $(\overline{h_t}, 1)$ to $1_3$ or $(h_i,3)$ with $i\in\{1,\cdots,t\}$. Again, since $T_{3,1}=S=\{1,h_i\mid 1\leq i\leq t\}$, we have that there is no arc from $1_3$ or $(h_i,3)$ to $(\overline{h_t},1)$.  Therefore, $(\overline{h_t},1)$ is an isolated vertex in $[\G_3^+(1_2)]$.
For $[\G_3^+(1_1)]$, since $T_{2,3}=S=\{1,h_i\mid 1\leq i
\leq t\}$,
we have that there is an arc from $1_2$ to $(h_1,3)$, an arc from $(h_3,2)$ to $(h_2h_3,3)$, and an arc from $(h_i,2)$ to $(h_1h_i,3)$ for each $i\in\{2,\cdots,t\}$. Thus there is no isolated vertex in $[\G_3^+(1_1)]$. Similarly, since
$T_{1,2}=S$, there is an arc from $1_1$ to $(h_1,2)$, an arc from $(h_3,1)$ to $(h_2h_3,2)$, and an arc from $(h_i,1)$ to $(h_1h_i,2)$ for each $i\in\{2,\cdots,t\}$, and so there is no isolated vertex in $[\G_3^+(1_3)]$. It follows that $[\G_3^+(1_2)]\ncong [\G_3^+(1_i)]$ with $i=1,3$, and so $A$ stabilizes $G_2$.
If there exists $\a\in A$ such that $1_1^{\a}=1_3$, then $\a$ maps  $[\G_3^+(1_1)]=[S\times \{2\}\cup L\times \{3\}]$ to $[\G_3^+(1_3)]=[S\times \{1\}\cup L\times \{2\}]$.
Note that $A$ stabilizes $G_2$. Then $\a$ maps $S\times \{2\}$ in $[\G_3^+(1_1)]$ to $L\times \{2\}$ in $[\G_3^+(1_3)]$ and so maps the arcs from $S\times \{2\}$ to $L\times \{3\}$ in $[\G_3^+(1_1)]$ to the arcs from $L\times \{2\}$ to $S\times \{1\}$ in $[\G_3^+(1_3)]$.
Note that there is an arc from $(s,2)$ to $L\times \{3\}$ for each $s\in S$. However, there is no arc from $(h_1h_{t-1},2)$ to $S\times \{1\}$ as $T_{2,1}=T=\{h_1h_t,\overline{h_i}\mid 1\leq i\leq t\}$ and $|h_1|\geq 3$. Therefore, $A$ stabilizes $G_1$ and $G_3$, respectively. It implies that $A$ stabilizes $G_i$ for each $i\in\{1,2,3\}$.\smallskip

Now, we are in the position to prove $A_{1_1}=1$. Since $[G_1\cup G_2]\cong \G_2$ is a 2-HOR of $G$, we have that $A_{1_1}$ fixes $G_1\cup G_2$ pointwise. If there exists $\beta \in A_{1_1}$ such that $(g,3)^{\beta}=(h,3)$ for some $g\neq h\in G$. Then $\beta$ maps $\G_3^+(g,3)\cap G_1$ to $\G_3^+(h,3)\cap G_1$,
that is, $\beta$ maps $Sg\times \{1\}$ to $Sh\times \{1\}$. Recall that $\beta \in A_{1_1}$ fixes $G_1$ pointwise. Then $Sg=Sh$ and so $g=h$, a contradiction.
Thus, $A_{1_1}$ fixes $G_3$ pointwise. It follows that $A_{1_1}$ fixes $G_1\cup G_2 \cup G_3$ pointwise and so $A_{1_1}=1$. Therefore, $A=R(G)A_{1_1}=R(G)$ and $\G_3$
is a 3-HOR of $G$.

\medskip
\f	{\bf Case 3:} $m\geq4$ is even
\medskip

Let $\widetilde {h_1}=h_2\cdots h_t$,
$\widetilde{h_t}=h_1\cdots h_{t-1}$ and $\widetilde{h_i}=h_1\cdots h_{i-1}h_{i+1}\cdots h_t$ with $2\leq i\leq t-1$.
Let $$M=\{h_1h_3,h_1h_4,\overline{h_i}\mid 2\leq i\leq t\} {\rm~and~} N=\{h_2h_3,h_2h_4,\widetilde{h_i}\mid1\leq i\leq t-1\}.$$ Then $|M|=|N|=t+1$ and $M\cap N^{-1}=\emptyset$. In the following, we will divided the proof into two subcases: $m=4$ and $m>4$. 

\medskip
\f	{\bf Subcase: 3.1} $m=4$
\medskip

Let $\G_4=\Cay(G,(T_{i,j})_{4\times 4})$ with $T_{1,2}=S$, $T_{2,1}=T$, $T_{3,4}=M$, $T_{4,3}=N$, $T_{1,3}=T_{1,4}=T_{2,3}=T_{2,4}=\{1\}$, $T_{3,2}=\{h_2\}$ and $T_{3,1}=T_{4,1}=T_{4,2}=\{h_1\}$. Then $\G_4$ is a 4-Haar oriented  graph of $G$ with valency $t+3$.
The left one in Figure~\ref{Fig3} might be of some help for understanding the structure of $\G_4$.
By the construction of $\G_4$, we have \begin{align*}
	\G_4^+(1_1)&=S\times \{2\}\cup \{1_3\}\cup \{1_4\}, \\
	\G_4^+(1_2)&=T\times \{1\}\cup\{1_3\}\cup \{1_4\}, \\
	\G_3^+(1_3)&=M\times \{4\}\cup (h_1,1)\cup(h_2,2) ,
	{\rm and }\\
		\G_4^+(1_4)&=N\times \{3\}\cup (h_1,1)\cup(h_1,2).
\end{align*}

\vspace{-2mm}

\begin{figure}	
		\centering
	\includegraphics[width=13cm]{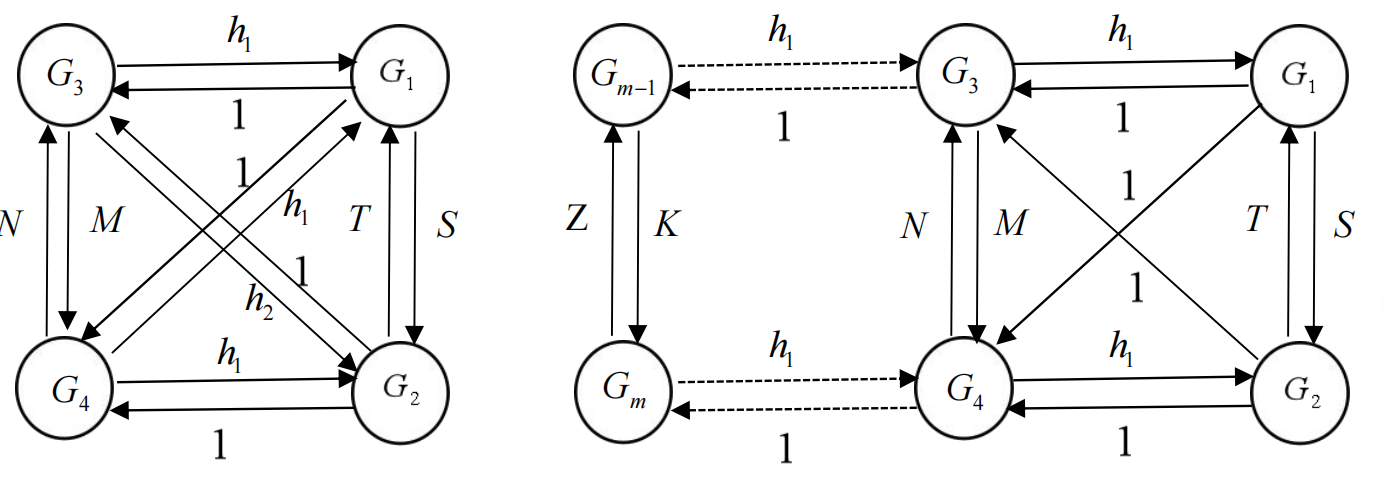}	
	\caption{\small{ The $4$-Haar oriented graph $\G_4$ and the $m$-Haar oriented graph $\G_m$ with $m\geq 6$ is even}}\label{Fig3}
\vspace{-2mm}
\end{figure}

Now we consider the induced sub-digraphs $[\G_4^+(1_i)]$ with $i\in \{1,\cdots,4\}$. Since $S=\{1,h_i\mid 1\leq i\leq t\}$, $T_{2,3}=T_{2,4}=\{1\}$, $T_{3,2}=T_{4,2}=\{h_1\}$, $T_{3,4}=M=\{h_1h_3,h_1h_4,\overline{h_i}\mid 2\leq i\leq t\}$, and $T_{4,3}=N=\{h_2h_3,h_2h_4,\widetilde{h_i}\mid 1\leq i\leq t-1\}$, we have that there are 4 arcs in $[\G_4^+(1_1)]$, that is, the arc from $1_2$ to $1_3$, the arc from $1_2$ to $1_4$, the arc from $1_3$ to $(h_2,2)$, and the
arc from $1_4$ to $(h_1,2)$. Similarly, since $T=\{h_1h_t,\overline{h_i}
\mid 1\leq i\leq t\}$, $T_{3,1}=T_{4,1}=\{h_1\}$, $T_{1,3}=T_{1,4}=\{1\}$, $T_{3,4}=M$ and $T_{4,3}=N$, we have that there are 2 arcs in $[\G_4^+(1_2)]$, that is, the arc from $1_3$ to $(h_1,2)$, and the
arc from $1_4$ to $(h_1,2)$.  Since $T_{4,1}=T_{4,2}=\{h_1\}$,
$T_{1,4}=T_{2,4}=\{1\}$ and $T_{1,2}=S$, we have that there is no arc in $[\G_4^+(1_3)]$. Again, since $T_{3,1}=\{h_1\}$, $T_{3,2}=\{h_2\}$,
$T_{1,3}=T_{2,3}=\{1\}$ and $T_{1,2}=S$, we have that there is an unique arc in $[\G_4^+(1_4)]$, that is, the arc from $(h_1,1)$ to $(h_1,2)$. Therefore, $[\G_4^+(1_i)]\ncong [\G_4^+(1_j)]$ with $i\neq j\in\{1,\cdots,4\}$.
Thus, $A$ stabilizes $G_i$ for each $i\in \{1,\cdots,4\}$.
Since $[G_1\cup G_2]\cong \G_2$ is a 2-HOR of $G$, we have $A_{1_1}$ fixes $G_1\cup G_2$ pointwise. Since $T_{1,3}=T_{1,4}=\{1\}$ and $A$ stabilizes $G_i$ with $i=3,4$, we have that $A_{1_1}$ fixes $G_3\cup G_4$ pointwise.
Therefore, $A_{1_1}=1$ and $\G_4$ is a 4-HOR of $G$.

\medskip
\f	{\bf Subcase 3.2:} $m>4$ is even
\medskip

Let $m>4$ be an even integer and
let $K=\{1,h_1h_2,h_i\mid 1\leq i\leq t\}$,
$Z=\{h_1h_t,\overline{h_t},\widetilde{h_i}\mid 1\leq i\leq t\}$. Then $|K|=|Z|=t+2$ and $K\cap Z^{-1}=\emptyset$.
Let  $\G_m=\Cay(G,(T_{i,j})_{m\times m})$ with
\begin{eqnarray*}
	T_{1,2}=S, \,\,\,\,\,\, T_{2,1}=T,\,\,\,\,\,\,\,\, T_{m-1,m}=K, \,\,\,\,\,\, T_{m,m-1}&=&Z; \\
	 T_{1,4}=T_{2,3}=T_{i,i+2}&=&\{1\}\,\,\,\,\,\,\,\,\,\,\,\,\,\,\,\,\,\,\mbox{ for } 1\leq i\leq m-2;\\
	T_{i,i-2}&=&\{h_1\},\,\,\,\,\,\,\,\,\,\,\,\,
	\mbox{ for } 3\leq i\leq m;\\
	 T_{i,i+1}&=&M,\,\,\,\,\,\,\,\,\,\,\,\,\,\,\,\,\,\,\mbox{ for } 3\leq i\leq m-3 \mbox{ is odd}; \\
	 	 T_{i,i-1}&=&N,\,\,\,\,\,\,\,\,\,\,\,\,\,\,\,\,\,\,\mbox{ for } 4\leq i\leq m-2 \mbox{ is even}; \\
	 T_{i,j}&=&\emptyset\,\,\,\,\,\,\,\,\,\,\,\,\,\,\,\,\,\,\,\,\,\,\,\,\mbox{ otherwise.}
\end{eqnarray*}
Since $S\cap T^{-1}=M\cap N^{-1}=K\cap Z^{-1}=\emptyset$ and  $|S|+2=|T|+2=|M|+2=|N|+2=|K|+1=|Z|+1=t+3$, we have that $\G_m$ is a connected $m$-Haar oriented graph of $G$ with valency $t+3$.
The right one in Figure~\ref{Fig3} might be of some help for understanding the structure of $\G_m$. Let $A=\Aut(\G_m)$. By the construction of $\G_m$, we have \begin{align*}
	\G_m^+(1_1)&=S\times \{2\}\cup \{1_3\}\cup \{1_4\}, \\
	\G_m^+(1_2)&=T\times \{1\}\cup\{1_3\}\cup \{1_4\}, \\
	\G_m^+(1_i)&=M\times \{i+1\}\cup (h_1,i-2)\cup \{1_{i+2}\} \mbox{ for } 3\leq i\leq m-3 \mbox{ is odd},\\
	\G_m^+(1_i)&=N\times \{i-1\}\cup \{(h_1,i-2)\}\cup \{1_{i+2}\}\mbox{ for } 4\leq i\leq m-2 \mbox{ is even},\\
		\G_m^+(1_{m-1})&=K\times \{m\}\cup \{(h_1,m-3)\},{\rm~and~}\\
				\G_m^+(1_{m})&=Z\times \{m-1\}\cup \{(h_1,m-2)\}.
\end{align*}
 In the following, we consider the induced sub-digraphs $[\G_m^+(1_i)]$ for each $i\in \{1,\cdots,m\}$. Similar as the Subcase 3.1, there are 3 arcs in $[\G_m^+(1_1)]$, that is, the arc from $1_2$ to $1_3$ and the arc from $1_2$ to $1_4$, and the arc $1_4$ to $(h_1,2)$; there is an unique arc in $[\G_m^+(1_2)]$, that is, the arc from $1_3$ to $(h_1,1)$, and there is no arc in $[\G_m^+(1_i)]$ for each $i\in \{3,\cdots,m\}$.
 Thus, $A$ stabilizes $G_1$ and $G_2$ respectively.
 Since $[G_1\cup G_2]\cong \G_2$ is a 2-HOR of $G$, we have $A_{1_1}$ fixes $G_1\cup G_2$ pointwise. Since $T_{1,3} = T_{1.4} = \{1\}$ and $T_{1,j} =\emptyset$ for each $j\geq 5$, it follows that $A$ stabilizes $G_3 \cup G_4$. If there exists $\a \in A$ such that $(g_3)^{\a}=h_4$ for some $g\neq h\in G$, then $\a$ maps $\G_m^+(g_3)=(Mg)_4\cup \{h_1g,1\}\cup \{g_5\}$ to $\G_m^+(h_4)=(Nh)_3\cup \{h_1h,2\}\cup \{h_6\}$. Since $A$ fixes $G_1$ and $G_2$ pointwise, this is impossible. Thus, $A$ stabilizes $G_3$ and $G_4$ setwise respectively.
Thus, $A_{1_1}$ stabilizes $G_3\cup G_4$ pointwise. Repeating this argument, we derive
 that $A$ stabilizes $G_i$ for each $i \in \{1,\cdots,m\}$, and $A_{1_1}$ fixes $G_1\cup \cdots \cup G_m$ pointwise.
 Therefore, $\G_m$ is an $m$-HOR of $G$.

\medskip
\f	{\bf Case 4:} $m\geq5$ is odd
\medskip

Let $m\geq 5$ be an odd integer and let $E=\{1,h_i\mid 1\leq i\leq t-1\}$ and
$F=\{\overline{h_i}\mid 1\leq i\leq t\}$. Then $|E|=|F|=t$ and $E\cap F^{-1}=\emptyset$. Let $\G_m=\Cay(G,(T_{i,j})_{m\times m})$ with
\begin{eqnarray*}
	T_{1,2}=T_{m-1,m}=T_{m,m-2}=S, \,\,\,\,\,\, T_{2,1}&=&T; \\
T_{i,i+2}&=&\{1\}\,\,\,\,\,\,\,\,\,\,\,\,\,\,\,\,\,\,\mbox{ for } 1\leq i\leq m-3;\\
	T_{m,m-1}=T_{m-2,m}=T_{i,i-2}&=&\{h_1\},\,\,\,\,\,\,\,\,\,\,\,\,
	\mbox{ for } 3\leq i\leq m-1;\\
	T_{i,i+1}&=&E,\,\,\,\,\,\,\,\,\,\,\,\,\,\,\,\,\,\,\mbox{ for } 3\leq i\leq m-2 \mbox{ is odd}; \\
	T_{i,i-1}&=&F,\,\,\,\,\,\,\,\,\,\,\,\,\,\,\,\,\,\,\mbox{ for } 4\leq i\leq m-3 \mbox{ is even}; \\
	 T_{i,j}&=&\emptyset\,\,\,\,\,\,\,\,\,\,\,\,\,\,\,\,\,\,\,\,\,\,\,\,\mbox{ otherwise.}
\end{eqnarray*}
Since $|h_1|\geq 3$, $S\cap T^{-1}=E\cap F^{-1}=\emptyset$ and  $|S|+1=|T|+1=|E|+2=|F|+2=t+2$, we have that $\G_m$ is a connected $m$-Haar oriented graph of $G$ with valency $t+2$.
 Figure~\ref{Fig4} might be of some help for understanding the structure of $\G_m$. Let $A=\Aut(\G_m)$. \smallskip

\begin{figure}
				\vspace{-2mm}
		\centering
	\includegraphics[width=10cm]{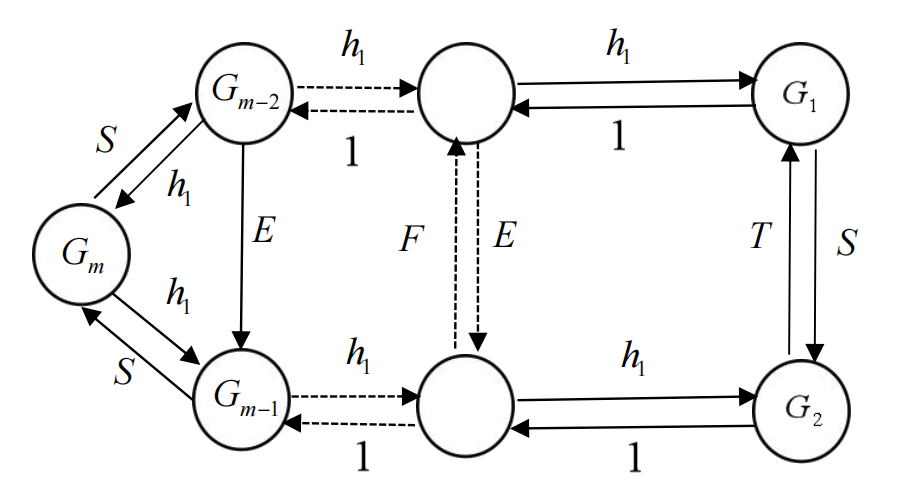}		
\vspace{-2mm}
	\caption{\small{The $m$-Haar oriented graph $\G_m$~with $m\geq 5$ is odd}}\label{Fig4}
\vspace{-2mm}	
\end{figure}

By the definition of $\G_m$, we have that $(1_{m-2},1_{m-1},1_m)$ is a directed 3-cycle as $1\in T_{m-2,m-1}\cap T_{m-1,m}\cap T_{m,m-2}$.
Hence, there is a 3-cycle through $g_i$ for each $g \in G$ and $i\in \{m-2,m-1,m\}$.
On the other hand, for each $i \in\{1,\cdots, m-3\}$, there is no 3-cycle through $g_i$ for any
$g \in G$. Hence $A$ stabilizes $G_{m-2} \cup G_{m-1}\cup G_m$. It follows that $A$ stabilizes $G_{m-4}\cup G_{m-3}$ as $T_{m-2,m-4}=T_{m-1,m-3}=\{h_1\}$ and $T_{m,j}=T_{m-1,j}=T_{m-2,j}=\emptyset$ for each $j\in \{1,\cdots,m-5\}$. Repeating this step, we have that $A$ stabilizes
$G_i\cup G_{i+1}$ for each $1\leq i\leq m-4$ with $i$ is odd.
In particular, $A$ stabilizes $G_1 \cup G_2$. Note that $[G_1\cup G_2]\cong \G_2$ is a 2-HOR of $G$. Therefore, $A$ stabilizes $G_1$ and $G_2$, respectively, and that $A_{1_1}$
fixes $G_1$ and $G_2$ pointwise.
 Repeating this argument, we derive
that $A$ stabilizes $G_i$ for each $i \in \{1,\cdots,m\}$, and $A_{1_1}$ fixes $G_1\cup \cdots\cup G_m$ pointwise. It follows that $\G_m$ is an $m$-HOR of $G$. This completes the proof.
	\end{proof}
	
\begin{proof}[\bf Proof of Theorem~$\ref{theo=main}$]
Let $m\geq 2$ be an integer and let $G$ be a finite group.  If $G$ is an elementary abelian $2$-group, by \cite[Theorem 1.1]{DFB},  $G$ has no $m$-HOR if and only if $m=2$, and $G\cong\mz_2^2$, $\mz_2^3$ or $\mz_2^4$;
$2\leq m\leq 3$, and $G\cong \mz_2$; or $2\leq m\leq 6$ and $G\cong\mz_1$. If $G$ is a non-elementary abelian $2$-group, then Theorem~\ref{theo=main} follows directly from Lemma~\ref{lem=m-HOR} for groups with at most three generators, and from Lemma~\ref{lem=m-HOR1} for those with at least four generators.    
\end{proof}
\section{$m$-partite oriented semiregular representation}

\begin{proof}[\bf Proof of Theorem~$\ref{theo=main1}$]
	Suppose $G$ has no $m$-POSR.
	Since an $m$-HOR is an $m$-POSR, Theorem~\ref{theo=main} implies that $G\cong Q_8$, $\mz_3$, $\mz_4$, $\mz_5$ or $\mz_2^n$ with $0\leq n\leq  4$. \smallskip

	 Assume $G$ is an elementary abelian $2$-group. Since an O$m$SR of an elementary 2-abelian group is an $m$-POSR, by \cite[Theorem 1.2]{DFB}, $G$ has no $m$-POSR if and only if $m=2$ and $G\cong \mz_2^2$ or $\mz_2^3$. \smallskip
	
	 Assume $G=\left\langle x,y\mid x^4=y^4=1,x^2=y^2, x^y=x^{-1}\right\rangle\cong  Q_8$. By Magma~\cite{magma}, $\Cay(G,(T_{i,j})_{2\times 2})$ with $T_{1,2}=\{1,x,y\}$ and $T_{2,1}=\{xy,x^2\}$ is a $2$-POSR of $G$.\smallskip
	
	 Assume $G=\left\langle x\mid x^k=1\right\rangle\cong  \mz_k$ with $3\leq k\leq 5$. By MAGMA~\cite{magma}, $\mz_3$ has no $2$-POSR,
	 $\Cay(G,(T_{i,j})_{2\times 2})$ with $T_{1,2}=\{1,x\}$ and $T_{2,1}=\{x^2\}$ is a $2$-POSR of $\mz_4$; and 	 $\Cay(G,(T_{i,j})_{2\times 2})$ with $T_{1,2}=\{1,x,x^{-1}\}$ and $T_{2,1}=\{x^2\}$ is a $2$-POSR of $\mz_5$.
\end{proof}
	
	\f {\bf Acknowledgement:}
	The author was supported by the National Natural Science Foundation of China (12571370, 12371025) and National Key R\&D Program of China (211070B62501).


\begin{thebibliography}{99}
		\bibitem{Babai}
		L. Babai, Finite digraphs with given regular automorphism groups,
		Period. Math. Hungar. 11 (1980), 257--270.
		
		
		
		\bibitem{magma}
		W. Bosma, C. Cannon, C. Playoust, The MAGMA algebra system I: The user language, J. Symbolic Comput. 24 (1997), 235--265.

		
			\bibitem{DFB}
		J.~-L.~Du, Y.~-Q.~Feng, S.~Bang, On oriented $m$-semiregular representations of finite groups, J. Graph Theory, 107 (2024), 485--508.
		
		
		
		\bibitem{DFS2020-1}
		J.~-L.~Du, Y.~-Q.~Feng, P.~Spiga, On Haar digraphical representations of groups, Discrete Math. 343 (2020), 112032: 1--6.
			
		\bibitem{DFS2022}
			J.~-L.~Du, Y.~-Q.~Feng, P.~Spiga,
		On $n$-partite digraphical representation of finite groups, J. Combin. Theory Ser. A 189 (2022), 105606: 1--13.
		
				\bibitem{DFXY}
		J.~-L.~Du, Y.~-Q.~Feng, B. Xia, D.-W. Yang, The existence of $m$-Haar graphical representations,
		J. Combin. Theory Ser. A 218 (2026), 106096: 1--30.
		
		
			\bibitem{DKY}
		J.~-L.~Du,
		Y. S. Kwon, F.-G. Yin, On $m$-partite oriented semiregular representations of ﬁnite
		groups generated by two elements, Discrete Math. 347 (2023), 114043:1--8.
		
		
	
		
		\bibitem{Godsil}
		C.D. Godsil, GRR's for non-solvable groups, in Algebraic Methods in Graph theory (Proc. Conf. Szeged 1978 L. Lov$\acute{a}$sz and V. T. Sos, eds), Coll. Math. Soc. J. Bolyai 25,
		North-Holland, Amsterdam, 1981, pp.221--239.
		
	

	\bibitem{Hetzel}
	D. Hetzel,  $\ddot{U}$ber regul$\ddot{a}$re graphische Darstellung von aufl$\ddot{o}$sbaren Gruppen, Technische Universit\"at, Berlin, 1976. (Diplomarbeit)
	
	\bibitem{Imrich}
	W. Imrich, Graphical regular representations of groups odd order, in: Combinatorics, Coll. Math. Soc. J\'anos. Bolayi 18 (1976), 611--621.
	
	
	\bibitem{ImrichWatkins2}
	W. Imrich, M. E. Watkins, On automorphism groups of Cayley graphs,
	Period. Math. Hungar. 7 (1976), 243--258.
	\bibitem{Konig}
	D. K$\ddot{o}$nig, Theory of finite and infinite graphs, translated from the German by Richard McCoart,
	with a commentary by W. T. Tutte and a biographical sketch by T. Gallai, Birkhauser Boston, Inc.,
	Boston, MA, 1990.
	
	
	
		
		\bibitem{ImrichWatkins}
		W. Imrich,  M.E. Watkins, On graphical regular representations of cyclic extensions of groups, Pac. J. Math. 55 (1974), 461--477.
		
	
		\bibitem{MorrisSpiga1}
		J. Morris, P. Spiga, Every finite non-solvable group admits an oriented regular representation, J. Combin. Theory Ser. B 126 (2017), 198--234.
		
		\bibitem{MorrisSpiga2}J. Morris, P.~Spiga, Classification of finite groups that admit an oriented regular
		representation, Bull. Lond. Math. Soc. 50 (2018), 811--831.
		
		
		\bibitem{MorrisSpiga3}J. Morris, P.~Spiga, Haar graphical representations ofﬁnite groups and an application to poset representations, J. Combin. Theory Ser. B 173 (2025), 279--304.
	
	
		\bibitem{NowitzWatkins1}
		L.~A.~Nowitz, M.~E.~Watkins, Graphical regular representations of non-abelian groups, $I$,
		Canad. J. Math. 24 (1972), 994--1008.
		
		\bibitem{NowitzWatkins2}
		L.~A.~Nowitz, M.~E.~Watkins, Graphical regular representations of non-abelian groups, $II$,
		Canad. J. Math. 24 (1972), 1009--1018.
		
		\bibitem{Spiga}
		P. Spiga, Finite groups admitting an oriented regular representation,
		J. Combin. Theory Ser. A 153 (2018), 76--97.
		
		
		
		
	\end{thebibliography}
\end{document}